\documentclass[11 pt]{article}
\usepackage{graphicx}
\usepackage{amsmath}
\usepackage{amssymb}
\usepackage{amsthm}
\usepackage{pxfonts}
\usepackage{enumerate}
\usepackage{color}
\usepackage{mathdots}
\usepackage{sectsty}
\usepackage[hidelinks]{hyperref}
\usepackage{tikz}
\usepackage{caption}
\usepackage{adjustbox}
\usepackage{fancyhdr}
\usepackage{verbatim}
\usepackage{url}

\sectionfont{\scshape\centering\fontsize{11}{14}\selectfont}
\subsectionfont{\scshape\fontsize{11}{14}\selectfont}

\newcommand\shorttitle{Partitions into Piatestki-Shapiro sequences}
\newcommand\authors{\small Nian Hong Zhou and  Ya-Li Li}

\hypersetup{
    colorlinks=true,       
    linkcolor=blue,          
    citecolor=cyan,        
}

\theoremstyle{plain}
\newtheorem{theorem}{Theorem}[section]
\newtheorem{lemma}[theorem]{Lemma}
\newtheorem{corollary}[theorem]{Corollary}
\newtheorem{proposition}[theorem]{Proposition}

\theoremstyle{remark}
\newtheorem{remark}{Remark}[section]

\makeatletter

\newcommand{\Rmnum}[1]{\expandafter\@slowromancap\romannumeral #1@}

\def\ri{\mathrm i}

\def\qb{\mathbb Q}
\def\rb{\mathbb R}
\def\nb{\mathbb N}
\def\zb{\mathbb Z}
\def\cb{{\mathbb C}}

\def\rrw{\rightarrow}
\def\res{\mathop{\rm{Res}}}
\numberwithin{equation}{section}

\title{\LARGE \bf Partitions into Piatetski-Shapiro sequences}
\author{
Nian Hong Zhou and Ya-Li Li\footnote{Corresponding author}
}

\date{}

\begin{document}

\maketitle




\begin{abstract}
Let $\kappa$ be a positive real number and $m\in\nb\cup\{\infty\}$ be given. Let $p_{\kappa, m}(n)$ denote the number of partitions
of $n$ into the parts from the Piatestki-Shapiro sequence $(\lfloor \ell^{\kappa}\rfloor)_{\ell\in \nb}$ with at most $m$ times (repetition allowed). In this paper we establish asymptotic formulas of Hardy-Ramanujan type for $p_{\kappa, m}(n)$, by employing a framework of asymptotics of partitions established by Roth-Szekeres in 1953, as well as some results on equidistribution.
\end{abstract}




\section{Introduction and statement of results}\label{sec1}

\subsection{Background}

Let $\lfloor u\rfloor$ denote the integral part of the real number $u$. The \emph{Piatetski-Shapiro sequence} of parameter $\kappa(\in\rb_+)$ (in short $PS(\kappa)$) is the sequence $(\lfloor \ell^{\kappa}\rfloor)_{\ell\in \nb}$, which is studied  in various directions. Let $m\in\nb\cup\{\infty\}$ be given and $p_{\kappa, m}(n)$ denote the number of partitions of $n$ into the sequence $PS(\kappa)$ with at most $m$ times (repetition allowed). We simply write $p_{\kappa}(n)=p_{\kappa,\infty}(n)$ and $q_{\kappa}(n)=p_{\kappa,1}(n)$. When $\kappa=1$, $p_{1}(n)$ is the well-known unrestricted partition function $p(n)$ and $q_1(n)$ is partition function $q(n)$ with partitions into distinct parts.

Since the time of Euler, we have known that $p_{\kappa,m}(n)$ has the following generating function
\begin{equation}
G_{\kappa,m}(z):=\sum_{n\ge 0}p_{\kappa, m}(n)e^{-nz}=\prod_{\ell\ge 1}\sum_{0\le r\le m}e^{-zr\lfloor \ell^{\kappa}\rfloor},
\end{equation}
where $\Re(z)>0$.

\medskip

Hardy and Ramanujan \cite{HR1918} established the following celebrated asymptotic formulas for $p(n)$ and $q(n)$:
\begin{equation}\label{eqhr1}
p(n)\sim \frac{1}{4\sqrt{3}n}e^{\pi\sqrt{2n/3}}\; \mbox{and}\; q(n)\sim \frac{1}{4\sqrt[4]{3} n^{3/4}}e^{\pi\sqrt{n/3}},
\end{equation}
as $n\rrw \infty$. We call these types asymptotic formulas  of Hardy-Ramanujan type. After Hardy and Ramanujan, such type problems have been widely investigated in many works in the literature. For example, Ingham \cite{MR5522}, Meinardus \cite{MR62781}, Schwarz \cite{ MR236135, MR254004} and Richmond \cite{Richmond1975, MR382210, MR412137}, Roth and Szekeres \cite{MR0067913}, Wright \cite{MR1555393} has investigated and established powerful asymptotic results for various types of integer partition functions.
\medskip

Of particular interest are functions related to partitions into certain special sequences.  As an
application of a powerful asymptotic results for general functions, Roth and Szekeres \cite{MR0067913} established asymptotic formulas for partition functions with partitions into rather general polynomial, whenever the argument of this polynomial taking from positive integers or primes. For more, see for examples \cite{MR3856854, MR4059127, MR3730445, MR3459558, MR3991428, Zhou1} and \cite{ MR2372796, MR1848510} for partitions into polynomials and primes, respectively. Erd\H{o}s and Richmond \cite{MR535018} considered the Hardy-Ramanujan type for partitions into the \emph{Beatty sequences} $(\lfloor \alpha\ell\rfloor)_{\ell\in\nb}$, where $\alpha>1$ is an irrational number with finite irrationality measure. The first author \cite{Zhou2} improved their result to any irrational number $\alpha>1$.

\medskip

 We now go back to this paper's topic, partitions into the sequence $SP(\kappa)$. Recently, Chen and the second author in \cite{MR3319122, MR3511994, MR3778641} investigated the asymptotics of $r$-th root partition function with any $r>1$, which corresponds to the partition function $p_{1/r}(n)$. They established upper and low bound and some asymptotics for $p_{1/r}(n)$. For $r=2$, Luca and Ralaivaosaona \cite{MR3531239} established an asymptotic formula of Hardy-Ramanujan type for $p_{1/2}(n)$, that is,
\begin{align}\label{LR}
p_{1/2}(n)&\sim 2^{5/18}3^{-1/2}\pi ^{-1/2}\zeta (3)^{7/18}e^{2\zeta '(-1)+\zeta '(0)}n^{-8/9}\nonumber\\
&\quad \times \exp \Bigg (\frac{3\zeta (3)^{1/3}}{2^{1/3}}n^{2/3}+\frac{\zeta (2)}{2^{2/3}\zeta (3)^{1/3}}n^{1/3}-\frac{\zeta (2)^2}{24\zeta (3)}\Bigg ),
\end{align}
as $n\rrw \infty$, where $\zeta(s)$ is the Riemann
zeta function. In \cite{MR4204766}, Chern gave the asymptotic formula of Hardy-Ramanujan type for $q_{1/2}(n)$, the square-root partitions into distinct parts, by adjusting the well-known approach of Meinardus \cite{MR62781}, that is,
\begin{align}\label{C}
q_{1/2}(n)&\sim 2^{-7/6}3^{-1/3}\pi ^{-1/2}\zeta (3)^{1/6}n^{-2/3}\nonumber\\&\quad \times \exp \Bigg (\frac{3^{4/3}\zeta (3)^{1/3}}{2}n^{2/3}+\frac{\zeta (2)}{2\cdot 3^{1/3}\zeta (3)^{1/3}}n^{1/3}-\frac{\zeta (2)^2}{72\zeta (3)}\Bigg ).
\end{align}
The asymptotic formula of Hardy-Ramanujan type for $k$-th root partition $p_{1/k}(n)$ with any $k\in\nb$ was established by Wu and the second author of this paper in a very recent work \cite{LW}.

\subsection{Main results}
In this paper, we shall investigate and establish the asymptotic formula of Hardy--Ramanujan type for $p_{\kappa, m}(n)$, for any $\kappa\in \rb_+$ and any $m\in\nb\cup \{\infty\}$. Throughout this paper, we write $\alpha=1/\kappa$. To state our main results, we need the following so-called \emph{Piatestki-Shapiro zeta function}:
\begin{equation}
\zeta_{\kappa}(s)=\sum_{n\ge 1}\frac{1}{\lfloor n^\kappa\rfloor^s}, \quad \Re(s)>\alpha.
\end{equation}
Furthermore, we need the following notations and definitions. Let ${\bf 1}_{event}$ be the indicator function.
For $u\in\rb$, let $\lceil u\rceil$ and $\{u\}$ denote the smallest integer $\ge u$ and the fractional part of $u$, respectively. For $z\in\cb$, let
$$E_{0}(z):=(1-z), \quad E_{h}(z)=(1-z)\exp\left(\sum_{1\le j\le h}z^j/j\right)(h\ge 1)$$
be the canonical factor. For any function $f$ and $x\in\rb$, let $f(\infty):=\lim\limits_{m\rrw +\infty}f(m)$  and $\widetilde{B}_1(x):=\{x\}-1/2$. We first prove the following proposition.

\begin{proposition}\label{pro1} For some real number $\sigma_{\kappa}\in(0,1]$, $\zeta_{\kappa}(s)$ can be monomorphic continued analytically to $\Re(s)>-\sigma_{\kappa}$ whose singularities are simple poles at
$s=\alpha, \alpha-1, \ldots, \alpha+1-\lceil \alpha\rceil,$
and their residues are
$$\res_{s=\alpha-h}\left(\zeta_\kappa(s)\right)=\frac{\Gamma(\alpha+1)}{(h+1)!\Gamma(\alpha-h)}, \quad h=0, 1, \cdots, \lceil \alpha\rceil-1.$$
Furthermore, for all $s$ with $\Re(s)\ge\varepsilon-\sigma_{\kappa}$ and $|\Im(s)|\ge 1$  with any $\varepsilon>0$, we have
$$\zeta_{\kappa}(s)\ll_\varepsilon |s|(|s|+1)^{\lfloor \alpha\rfloor+1}.$$
Moreover,
$$\zeta_{\kappa}(0)=-\frac{{\bf 1}_{\alpha\not\in\nb}}{2}-\frac{{\bf 1}_{\alpha\in\nb}}{1+\kappa},$$
and
$$\zeta_{\kappa}'(0)=\begin{cases}
\qquad\qquad\qquad \sum\limits_{0\le h<\alpha}\binom{\alpha}{h}\zeta'(-h), &\alpha\in\nb,\\
\sum\limits_{1\le h\le \lceil\alpha\rceil}\frac{\zeta(h-\alpha)-1}{h}+\sum\limits_{n\ge 2}\log\left(\frac{E_0(1/n)^{\widetilde{B}_1(n^{\alpha})+{\bf 1}_{n^{\alpha}\in\nb}}}{E_{\lceil\alpha\rceil}(1/n)^{n^{\alpha}}}\right), &\alpha\not\in\nb.
\end{cases}$$
\end{proposition}

\begin{remark}Since for $k\in \nb$, we have $\zeta_{k}(s)=\zeta(k s)$. Note that $2\zeta'(0)=-\log(2\pi)$, by the above expression of $\zeta_{\kappa}'(0)$ we obtain that
$$-\frac k 2\log(2\pi)+1=\zeta(1-1/k)+\sum_{n\ge 2}\log\left(\frac{E_0(1/n^{k})E_0(1/n)^{\widetilde{B}_1(n^{1/k})}}{E_{1}(1/n)^{n^{1/k}}}\right),$$
for all integers $k\ge 2$. For $k=2,3$, the above identities has passed the numerical examination of {\bf Mathematica}.
\end{remark}
Under the above proposition, our main results are stated as follows.
\begin{theorem}\label{mth1}For any $m\in\nb\cup\{\infty\}$, as $n\rrw \infty$
\begin{align*}
p_{\kappa, m}(n)\sim \lambda_{\kappa,m,c}\left(\frac{\beta_{\kappa,m}}{n}\right)^{\delta_{\kappa,m}}\exp\left(\sum_{0\le h\le \alpha}\lambda_{\kappa,m}(h)\left(\frac{n}{\beta_{\kappa,m}}\right)^{\frac{\alpha-h}{\alpha+1}}\right),
\end{align*}
where $\delta_{\kappa,m}=\frac{1}{2}+\frac{1/2-\zeta_{\kappa}(0){\bf 1}_{m=\infty}}{1+\alpha}$,
$$(\beta_{\kappa,m}, \lambda_{\kappa,m,c})=\left(\frac{\left(1-(m+1)^{-\alpha}\right)\zeta(\alpha+1)}{\alpha^{-1}\Gamma(\alpha+1)^{-1}}, \frac{(m+1)^{\zeta_{\kappa}(0)}{\bf 1}_{m\in\nb}+e^{\zeta_{\kappa}'(0)}{\bf 1}_{m=\infty}}{\sqrt{2\pi \beta_{\kappa,m}(\alpha+1)}}\right),
$$
and $\lambda_{\kappa,m}(h)$ are computable real constants.
\end{theorem}

\begin{remark} For the details of the calculation of $\lambda_{\kappa, m}(h)$,   see Lemma \ref{lm} of Section \ref{sec2}. In particular, with the help of {\bf Mathematica}, we can get
$$
\left(\lambda_{\kappa,m}(0), \lambda_{\kappa,m}(1)\right)=\left((1+\kappa)\beta_{\kappa,m}, \frac{(1-(m+1)^{1-\alpha})\zeta(\alpha){\bf 1}_{\alpha>1}}{2\Gamma(1+\alpha)^{-1}}\right),
$$
and
\begin{align*}
\begin{cases}\lambda_{\kappa,m}(2)=\frac{(1-(m+1)^{2-\alpha})\zeta(\alpha-1){\bf 1}_{\alpha> 2}}{6\Gamma(\alpha+1)^{-1}}-\frac{(\alpha-1)^2\lambda_{\kappa,m}(1)^2}{2\alpha\lambda_{\kappa,m}(0)},\\
\lambda_{\kappa,m}(3)=\frac{(1-(m+1)^{3-\alpha}){\bf 1}_{\alpha>3}}{24\zeta(\alpha-2)^{-1}\Gamma(\alpha+1)^{-1}}-\frac{(\alpha-1)(\alpha-2)\lambda_{\kappa,m}(1)\lambda_{\kappa,m}(2)}{\alpha\lambda_{\kappa,m}(0)}-\frac{(\alpha-4)(\alpha-1)^3\lambda_{\kappa,m}(1)^3}{6\alpha^2\lambda_{\kappa,m}(0)^2}.
\end{cases}
\end{align*}
\end{remark}

For the case $\kappa>1$, Theorem \ref{mth1} confirms a conjecture of Chen and the second author of this paper stated as in \cite{MR3778641}. We can get the following conclusion.
\begin{corollary}\label{mcor1}
Let $\kappa>1$. As $n\rrw \infty$
  \begin{align*}
p_{\kappa}(n)\sim c_\kappa n^{-\frac{\alpha+3}{2(1+\alpha)}}\exp\left((1+\kappa)\beta_{\kappa}^{\frac{1}{1+\alpha}}n^{\frac{1}{1+\kappa}}\right),
\end{align*} and
  \begin{align*}
q_{\kappa}(n)\sim c_{\kappa1} n^{-\frac{\alpha+2}{2(1+\alpha)}}\exp\left((1+\kappa)\left((1-2^{-\alpha})\beta_{\kappa}\right)^{\frac{1}{1+\alpha}}n^{\frac{1}{1+\kappa}}\right),
\end{align*}
where $\beta_\kappa=\alpha\Gamma(\alpha+1)\zeta(\alpha+1)$ and
$$(c_\kappa,c_{\kappa1})=\frac{\left(\sqrt{2}e^{\zeta_\kappa'(0)}\beta_{\kappa}^{\frac{1}{1+\alpha}}, \left(\left(1-2^{-\alpha}\right)\beta_{\kappa}\right)^{\frac{1}{2(1+\alpha)}}\right)}{2\sqrt{\pi(1+\alpha)}}.$$
\end{corollary}

Using Theorem \ref{mth1},  we can find the same asymptotics as Luca-Ralaivaosaona \cite{MR3531239} and Chern \cite{MR4204766}, that is, asymptotics of $p_{1/2}(n)$ and $q_{1/2}(n)$ are \eqref{LR} and \eqref{C}, respectively.
Moreover, by Theorem \ref{mth1} we can also find the same leading asymptotic formulas as Wu and the second author of this paper in \cite{LW}, for the $k$-th root partition functions $p_{1/k}(n)$ \footnote{in \cite{LW}, some coefficients of asymptotic formula for $p_{1/k}(n)$ are not
explicitly given.}. In particular, by noting $\zeta(2)=\pi^2/6$, $\zeta(4)=\pi^4/90$ and $\zeta'(-2)=\zeta(3)/4\pi^2$, we have the following exactly asymptotic formulas.

\begin{corollary}\label{mcor3}As $n\rrw \infty$
\begin{align*}
p_{1/3}(n)\sim &\frac{(25\pi)^{1/4}e^{\frac{225 \zeta (3)^3}{\pi ^8}-\frac{2 \zeta (3)}{\pi ^2}+3\zeta'(-1)}}{4(5n)^{13/16}}\\
&\times\exp \left(\frac{4 \pi  (5n)^{3/4}}{15}+\frac{3\zeta (3)(5n)^{1/2}}{\pi ^2}+\frac{\left(\pi ^6-135 \zeta (3)^2\right)(5n)^{1/4}}{6 \pi ^5}\right),
\end{align*}
and
\begin{align*}
q_{1/3}(n)\sim&\frac{(5/14)^{1/2}e^{\frac{6075 \zeta (3)^3}{49 \pi ^8}-\frac{15 \zeta (3)}{28 \pi ^2}}}{2^{11/4}(5n/14)^{5/8}}\\
&\times\exp \left(\frac{28\pi\big(\frac{5n}{14}\big)^{3/4}}{15}+\frac{9\zeta (3)\big(\frac{5n}{14}\big)^{1/2}}{\pi ^2}+\frac{\left(7 \pi ^6-1215 \zeta (3)^2\right)\big(\frac{5n}{14}\big)^{1/4} }{42 \pi ^5}\right).
\end{align*}
\end{corollary}

Our paper is organized as follows: In Section \ref{sec2} we will prove
Theorem \ref{mth1}, thus establish an asymptotic formula $p_{\kappa,m}(n)$.
In Section~\ref{sec3} we use some results on equidistribution to check the sequence $PS(\kappa)$ meets the conditions of the framework of Roth and Szekeres \cite{MR0067913}. In Section~\ref{sec4} we use a theorem of van der Corput \cite{MR1544956}, the classical Taylor's theorem to prove Proportion \ref{pro1}, that is, the analytic continuation for the Piatestki-Shapiro zeta function $\zeta_\kappa(s)$.

\section{The proof of Theorem \ref{mth1}}\label{sec2}
In this section we give the proof of Theorem \ref{mth1}. Our proof is based on the framework of Roth and Szekeres \cite{MR0067913} and its generalization, that is, the work of Liardet and Thomas \cite{MR3330365}.  In order to be able to use this framework, we need to prove that the sequence $PS(\kappa)$ meets the following conditions.
\begin{enumerate}[(I)]
  \item \label{1} There exist $\mu>0$ such that
  $$ \mu=\lim_{H\rrw \infty}\frac{\log \lfloor H^{\kappa}\rfloor}{\log H},$$
  \item \label{2} We have\footnote{Note that the condition corresponding to $p_{\kappa,\infty}(n)$ in Roth and Szekeres \cite[p.259, (II*)]{MR0067913} is $$\lim_{H\rrw \infty}\inf_{\frac{1}{2\lfloor H^{\kappa}\rfloor}\le y\le \frac{1}{2}}\left(\frac{1}{\log H}\sum_{1\le \ell\le H}\left(\frac{\lfloor H^\kappa \rfloor}{\lfloor \ell^\kappa \rfloor}\right)^2\|\lfloor \ell^\kappa\rfloor y\|^2\right)=\infty.
$$
Obviously, this condition is weaker than that required by $p_{\kappa,m}(n)$ with $m\in\nb$.
}
  $$\lim_{H\rrw \infty}\inf_{\frac{1}{2\lfloor H^{\kappa}\rfloor}\le y\le \frac{1}{2}}\left(\frac{1}{\log H}\sum_{1\le \ell\le H}\|\lfloor \ell^\kappa\rfloor y\|^2\right)=\infty.
$$
\end{enumerate}

Clearly, the condition (\ref{1}) is easily to verified. For the condition (\ref{2}), we shall check it in Section \ref{sec3}, by employ the equidistribution properties of the sequence $(\{\ell^{\kappa}\})_{\ell\ge 1}$, as well as some works of van der Corput \cite{MR1544956}.

Define that for any $\kappa>0$
$$L_{\kappa}(z)=-\sum_{\ell\ge 1}\log\left(1-e^{-z\lfloor \ell^\kappa\rfloor}\right), \quad \Re(z)>0.$$
Then, since the sequence $PS(\kappa)$ satisfies the above two conditions, following the works of Roth and Szekeres \cite[Theorem 2]{MR0067913} and Liardet and Thomas \cite[Theorem 14.2]{MR3330365}, we can get the following proposition immediately.
\begin{proposition}\label{thm1}Let $m\in\nb$. For a small positive number $\varepsilon$,  as $n\rrw +\infty$
\begin{equation*}
p_{\kappa}(n)=\frac{\exp\left(L_{\kappa}(x)+nx\right)}{\sqrt{2\pi L_{\kappa}''(x)}}\left(1+O
\left(n^{-\frac{1}{\mu+1}+\varepsilon}\right)\right),
\end{equation*}
and
\begin{equation*}
p_{\kappa,m}(n)=\frac{\exp\left(L_{\kappa}(y)-L_{\kappa}((m+1)y)+ny\right)}{\sqrt{2\pi (L_{\kappa}''(y)-(m+1)^2L_{\kappa}''((m+1)y))}}\left(1+O
\left(n^{-\frac{1}{\mu+1}+\varepsilon}\right)\right),
\end{equation*}
where $x$ and $y$ are positive numbers which solves the equations
$$L_{\kappa}'(x)+n=0\;\; \mbox{and}\; L_{\kappa}'(y)-(m+1)L_{\kappa}'((m+1)y)+n=0,$$
respectively.
\end{proposition}
According to the Mellin transform representation of $L_{\kappa}(x)$ and the analytic properties of $\zeta_{\kappa}(s)$, that is, Proposition \ref{pro1}, as well as some well fact on analytic number theory, we will prove the following proposition.
\begin{proposition}\label{pro3}Let $\varepsilon$ be a sufficiently small positive number. For some real number $\sigma_{\kappa}\in(0,1]$ which is in Proposition \ref{pro1} and each integer $j\ge 0$, as $x\rrw 0^+$
\begin{align*}
\frac{\, d^j }{\,d x^j}L_{\kappa}(x)=&\frac{\, d^j }{\,d x^j}\left(\sum_{0\le h<\alpha }\frac{\Gamma(\alpha+1)\zeta(\alpha-h+1)}{(h+1)!x^{\alpha-h}}-\zeta_{\kappa}(0)\log x+\zeta_{\kappa}'(0)\right)+O(x^{\sigma_\kappa-\varepsilon-j}).
\end{align*}
\end{proposition}

\begin{proof}By the Mellin's inversion formula, for any $c>\alpha$ we have
\begin{align*}
L_{\kappa}(z)&=\sum_{\ell\ge 1}\sum_{n\ge 1}\frac{e^{-zn \lfloor \ell^{\kappa}\rfloor}}{n}\\
&=\frac{1}{2\pi\ri}\int_{c-\ri\infty}^{c+\ri\infty}\Gamma(s)\zeta(s+1)\zeta_{\kappa}(s)z^{-s}\,ds.
\end{align*}
Therefore, for each integer $j\ge 0$,
\begin{align}\label{eqIg2}
(-1)^jL_{\kappa}^{(j)}(z)&=\frac{1}{2\pi\ri}\int_{c-\ri\infty}^{c+\ri\infty}\Gamma(s+j)\zeta(s+1)\zeta_{\kappa}(s)z^{-s-j}\,ds.
\end{align}
It is well known (see \cite[p.38, p.92]{MR0364103}) that fixed any $a,b\in\rb$, for all $\sigma\in[a, b]$ and any $|t|\ge 1$,
$$\Gamma(\sigma+\ri t)\ll |t|^{\sigma-1/2}\exp\left(-\frac{\pi}{2}|t|\right)$$
and
$$
\zeta(\sigma+\ri t)\ll |t|^{|\sigma|+1/2}.
$$
From the estimation of $\zeta_{\kappa}(s)$ in Proposition \ref{pro1} we have $$\Gamma(s)\zeta(s+1)\zeta_{\kappa}(s)\ll_\varepsilon |t|^{O(1)}\exp\left(-\frac{\pi}{2}|t|\right),$$
for $t\in \rb\setminus (-1,1)$ and $\sigma\ge\varepsilon-\sigma_{\kappa}$. Thus, by the residue theorem, we can move the line of integration \eqref{eqIg2} to the $\Re(s)=\varepsilon-\sigma_\kappa$ and taking into account the possible poles on $\Re(s)>-\sigma_\kappa$ of $\Gamma(s)\zeta(s+1)\zeta_{\kappa}(s)z^{-s-j}$, we obtain
\begin{align*}
(-1)^jL_{\kappa}^{(j)}(z)&=\res_{\Re(s)>-\sigma_\kappa}\left(\Gamma(s+j)\zeta(s+1)\zeta_{\kappa}(s)z^{-s-j}\right)
+O_{\varepsilon}\left(|z|^{-j+\sigma_\kappa-\varepsilon}\right).
\end{align*}
Note that the only poles of gamma function $\Gamma(s)$ are at $s=-k~(k\in\zb_{\ge 0})$,
and all are simple; $s=1$ is the only pole of $\zeta(s)$ and is simple; from Proposition \ref{pro1}, all poles of $\zeta_\kappa(s)$ lies on $\Re(s)>-\sigma_{\kappa}$ are simple listed as follows
$$\alpha, \alpha-1, \ldots, \alpha+1-\lceil \alpha\rceil,$$
and for all integer $h$ with $0\le h<\alpha$,
$$\res_{s=\alpha-h}\left(\zeta_\kappa(s)\right)=\frac{\Gamma(\alpha+1)}{(h+1)!\Gamma(\alpha-h)}.$$
Thus for $j=0$, we have
\begin{align*}
\res_{\Re(s)>-\sigma_\kappa}\left(\frac{\Gamma(s)\zeta(s+1)\zeta_{\kappa}(s)}{z^{s}}\right)=&\res_{s=0}\left(\Gamma(s)\zeta(s+1)\zeta_{\kappa}(s)z^{-s}\right)\\
&+\sum_{0\le h<\alpha}\res_{s=\alpha-h}\left(\Gamma(s)\zeta(s+1)\zeta_{\kappa}(s)z^{-s}\right)\\
=&\zeta_{\kappa}'(0)-\zeta_{\kappa}(0)\log z+\sum_{0\le h<\alpha}\frac{\Gamma(\alpha+1)\zeta(\alpha-h+1)}{(h+1)!z^{\alpha-h}},
\end{align*}
and for integer $j\ge 1$, we have
\begin{align*}
\res_{\Re(s)>-\sigma_\kappa}\left(\frac{\Gamma(s+j)\zeta(s+1)\zeta_{\kappa}(s)}{z^{s+j}}\right)=&\frac{\Gamma(j)\zeta_{\kappa}(0)}{z^j}\\
&+\sum_{0\le h<\alpha}\frac{\Gamma(\alpha+1)\Gamma(\alpha-h+j)\zeta(\alpha-h+1)}{(h+1)!\Gamma(\alpha-h)z^{\alpha-h+j}}.
\end{align*}
This completes the proof.
\end{proof}

Under Proposition \ref{thm1}, the proof of Theorem \ref{mth1} immediately follows from the sharp asymptotics of $L_{\kappa}(z)$ stated in  Proposition \ref{pro3} and the following Lemma \ref{lm}.
\begin{lemma}\label{lm}Let $a,b\in\rb$ and let $(c_h)_{h\ge0}$ be a sequence of real numbers which has finite support. Let $\alpha, c_0\in\rb_+$ and $\sigma\in(0,1]$ and  real function $L(t)$ satisfy
\begin{equation}\label{eqs}
\frac{\, d^j }{\,d t^j} L(t)=\frac{\, d^j }{\,d t^j}\bigg(a-b\log t+\sum_{h\ge 0}c_ht^{h-\alpha}\bigg)+O(t^{\sigma-j}),\; j=0,1,2
\end{equation}
as $t\rrw 0^+$. If $u>0$ solves the equation~$L'(u)+n=0$, then as $n\rrw +\infty$ there exist computable real constants $\lambda_h$ such that
\begin{equation*}
\frac{e^{L(u)+nu}}{\sqrt{2\pi L''(u)}}= \frac{\left(\frac{c_0\alpha}{n}\right)^{\frac{1}{2}+\frac{1/2-b}{\alpha+1}}e^{a}}{\sqrt{2\pi c_0\alpha(\alpha+1)}}\exp\left(\sum_{0\le h\le \alpha}\lambda_h\left(\frac{n}{c_0\alpha}\right)^{\frac{\alpha-h}{\alpha+1}}+O\bigg(n^{-\frac{\min(1-\{\alpha\},\sigma, \alpha)}{1+\alpha}}\bigg)\right).
\end{equation*}
In particular,
$$
(\lambda_0, \lambda_2)=\left(c_0(\alpha+1), c_2-\frac{(\alpha-1)^2\lambda_1^2}{2\alpha\lambda_0}\right)
$$
and
$$
(\lambda_1, \lambda_3)=\left(c_1, c_3-\frac{(\alpha-1)(\alpha-2)\lambda_1\lambda_2}{\alpha\lambda_0}-\frac{(\alpha-4)(\alpha-1)^3\lambda_1^3}{6\alpha^2\lambda_0^2}\right).
$$
\end{lemma}
\begin{proof}
Let $n\rrw +\infty$. From  $n+L'(u)=0$ and $c_0>0$, it is clear that $u\rrw 0^+$ and
$$
n=\sum_{h\ge 0}\frac{(\alpha-h)c_h}{u^{\alpha+1-h}}+\frac{b}{u}+O\left(u^{-1+\sigma}\right).
$$
This means $u\sim (\alpha c_0/n)^{\frac{1}{1+\alpha}}$. By \eqref{eqs} it implies $L''(u)=c_0\alpha(\alpha+1)u^{-\alpha-2}(1+O(u))$. By further reduction formula we obtain
\begin{equation}\label{eqp}
\frac{e^{L(u)+nu}}{\sqrt{2\pi L''(u)}}=u^{1+\alpha/2-b}(1+O(u))\frac{\exp\left(b+\sum_{h\ge 0}c_h(\alpha-h+1)u^{h-\alpha}\right)}{e^{-a}\sqrt{2\pi c_0\alpha(\alpha+1) }}.
\end{equation}
Let $P(u)=\sum_{h\ge 1}\frac{(\alpha-h)c_h}{\alpha c_0}u^{h-1}$ and $t_n=\left({\alpha c_0}/{n}\right)^{\frac{1}{1+\alpha}}$. Using generalized binomial theorem we obtain
\begin{align*}
t_n=&u
\left(1+uP(u)+{bu^{\alpha}}/{\alpha c_0}+O\left(u^{\alpha+\sigma}\right)\right)^{-\frac{1}{1+\alpha}}\nonumber\\
&=u
\left(1+uP(u)\right)^{-\frac{1}{1+\alpha}}-\frac{u^{\alpha+1}}{c_0\alpha(\alpha+1)}\left(b+O\left(u^{\min(\alpha, \sigma)}\right)\right).
\end{align*}
This means that
\begin{align}\label{eqmm0}
u=t_n\big(1+O\big(t_n^{\min(1,\alpha)}\big)\big).
\end{align}
Denoting by
$
\hat{t}_n=u\left(1+uP(u)\right)^{-\frac{1}{1+\alpha}},
$
then
\begin{align}\label{eqtn}
\hat{t}_n&=t_n+u^{\alpha+1}\left(b+O\left(u^{\min(\alpha,\sigma)}\right)\right)/c_0\alpha(\alpha+1)\nonumber\\
&=t_n+bt_n^{\alpha+1}/c_0\alpha(\alpha+1)+O\left(t_n^{\alpha+1+\min(\alpha,\sigma)}\right).
\end{align}
Since $(c_h)_h$ have finite support, $P(u)$ is a polynomial and hence $(1+zP(z))^{1/(1+\alpha)}$
is analytic at $z=0$. Recall that the B\"{u}rmann Theorem (see \cite[p.~129]{MR1424469}) states that: Suppose that $f(z)$ is analytic at $z=0$ and $f(0)=0$, if $t=uf(u)^{-1}$ then for $t\rrw 0$
$$u=\sum_{j\ge 0}\frac{t^{j+1}}{(j+1) !}\frac{\,d^j}{\,dz^j}\bigg|_{z=0}\big( f(z)^{j+1}\big).$$
Therefore, as $\hat{t}_n\rrw 0$,
\begin{equation*}
u=\sum_{j\ge 0}\frac{\hat{t}_n^{j+1}}{(j+1) !}\frac{\,d^j}{\,dz^j}\bigg|_{z=0}\big( (1+zP(z))^{\frac{j+1}{\alpha+1}}\big).
\end{equation*}
Then, for each integer $h\ge 0$ it is clear that $(\hat{t}_nu)^{h-\alpha}$ is analytic at $\hat{t}_n=0$. Hence
the finite sum function $\sum_{h\ge 0}c_h(\alpha-h+1)u^{h-\alpha}$ is analytic at $\hat{t}_n=0$. So, as $\hat{t}_n\rrw 0$,
\begin{align}\label{eqtn1}
\sum_{h\ge 0}&c_h(\alpha-h+1)u^{h-\alpha}= \hat{t}_n^{-\alpha}\sum_{j\ge 0}\lambda_j\hat{t}_n^{j}
\end{align}
holds for a computable sequence $(\lambda_j)_j$ of real numbers which can be obtained with the help of {\bf Mathematica}.
Inserting \eqref{eqtn} and using the generalized binomial theorem, we have
$$\hat{t}_{n}^{j-\alpha}=t_n^{j-\alpha}\left(1+(j-\alpha)bt_n^{\alpha}/c_0\alpha(\alpha+1)+O\left(t_n^{\alpha+\min(\alpha,\sigma)}\right)\right).$$
Inserting \eqref{eqtn1} we have
\begin{align}\label{eqmm1}
\sum_{h\ge 0}&c_h(\alpha-h+1)u^{h-\alpha}= t_n^{-\alpha}\sum_{0\le j\le \alpha}\lambda_jt_n^{j}-b+O\left(t_n^{\lfloor\alpha\rfloor+1-\alpha}+t_n^{\min(\alpha,\sigma)}\right)
\end{align}
Substituting \eqref{eqmm1} and $t_n=(c_0\alpha/n)^{\frac{1}{\alpha+1}}$ to \eqref{eqp}, we complete the proof.
\end{proof}

\section{The check of the condition (\ref{2})}\label{sec3}
For the case of $\kappa\in\nb$, Roth and Szekeres \cite[P.241]{MR0067913} have mentioned that for any integer $\kappa>0$,  $(\ell^\kappa)_{\ell\in\nb}$ satisfies the condition (\ref{2}). Hence we only verify the case of $\kappa\in\rb_+\setminus \nb$.  We begin with the following lemmas.
\begin{lemma}\label{lem3.1} For $\kappa\in\rb_+\setminus \nb$ and $|y|\le 1/2$, we have
$$
\left|\sum_{H/2\le \ell<H}e^{2\pi\ri y\lfloor\ell^{\kappa}\rfloor}\right|\le \bigg|\sum_{H/2\le \ell<H}e^{2\pi\ri y\ell^{\kappa}}\bigg|+\frac{\pi H}{8}+o(H),
$$as $H\rrw +\infty$.
\end{lemma}
\begin{proof}
By  the triangle inequality, we have
\begin{align*}
\bigg|\sum_{H/2\le \ell<H}e^{2\pi\ri y\lfloor\ell^{\kappa}\rfloor}\bigg|&\le \bigg|\sum_{H/2\le
\ell<H}e^{2\pi\ri y(\ell^{\kappa}-\frac 12)}\bigg|+\bigg|\sum_{H/2\le \ell<H}e^{2\pi\ri
y(\ell^{\kappa}-\frac 12)}\left(e^{2\pi\ri y(\frac 12-\{\ell^{\kappa}\})}-1\right)\bigg|\\
&\le \bigg|\sum_{H/2\le \ell<H}e^{2\pi\ri y\ell^{\kappa}}\bigg|+\sum_{H/2\le \ell<H}\bigg|e^{2\pi\ri
y(1/2-\{\ell^{\kappa}\})}-1\bigg|.
\end{align*}
Notice that
\begin{align*}
\bigg|e^{2\pi\ri y(1/2-\{\ell^{\kappa}\})}-1\bigg|&=\bigg|\int_{0}^{2\pi y(1/2-\{\ell^{\kappa}\})}e^{\ri u}\,du\bigg|\\
&\le \bigg|2\pi y(1/2-\{\ell^{\kappa}\})\bigg|\le \pi\left|1/2-\{\ell^{\kappa}\}\right|,
\end{align*}
and $(\{\ell^{\kappa}\})_{\ell\in\nb}$ is uniformly distributed in $(0,1)$ for $\kappa\in\rb_+\setminus \nb$, see for example \cite[p.31, Exercises 3.9]{MR0419394}. It follows that
\begin{align*}
\sum_{H/2\le \ell<H}\bigg|e^{2\pi\ri y(1/2-\{\ell^{\kappa}\})}-1\bigg|&\le \pi\sum_{H/2\le \ell<H}\left|1/2-\{\ell^{\kappa}\}\right|\\
&\sim \frac {\pi H} 2\int_{0}^{1}|1/2-x|\,dx=\frac{\pi H}{8},
\end{align*}
as $H\rrw \infty$. Thus, we complete the proof of Lemma \ref{lem3.1}.
\end{proof}
\begin{lemma}\label{lem3.2} For $\kappa\in\rb_+\setminus \nb$ and $(\log H)^{-2\kappa} \ll y\le 1/2$, we have
$$\sum_{H/2\le \ell<H}e^{2\pi\ri y\ell^{\kappa}}=o(H),$$
as $H\rrw \infty$.
\end{lemma}
\begin{proof}
For the cases of $\kappa\in(0,1]$, it follows from \cite[p.206, Lemma 8.8]{MR2061214} and $(\log H)^{-2\kappa} \ll y\le 1/2$ that
\begin{align*}
\sum_{H/2\le \ell<H}e^{2\pi\ri y\ell^{\kappa}}&=\int_{H/2}^{H}e^{2\pi\ri yt^{\kappa}}dt+O\left(\frac 1 {1-yH^{\kappa-1}}\right)\\
&=\frac{1}{\kappa}\int_{(H/2)^\kappa}^{H^\kappa}e^{2\pi\ri yu}u^{1/\kappa-1}du +O(1)\\
&=\frac 1 {2\pi\ri y\kappa}\int_{(H/2)^\kappa}^{H^\kappa}u^{1/\kappa-1}d e^{2\pi\ri yu}+O(1)\\
&=O\left(y^{-1} H^{1-\kappa}\right)+O(1)=o(H).
\end{align*}
For the cases of $\kappa\in(1,\infty)$, we have $\lceil\kappa\rceil\ge 2$. It follows from \cite[p.213, Theorem 8.20]{MR2061214} and $(\log H)^{-2\kappa} \ll y\le 1/2$ that
\begin{align*}
\sum_{H/2\le \ell<H}e^{2\pi\ri y\ell^{\kappa}}&\ll H^{1-2^{2-k}}(yH^{\kappa-k})^{-\frac{1}{2^k-2}}+H(y H^{\kappa-k})^{\frac{1}{2^k-2}}\\
&=H\left( y H^{\kappa-k+4-8\cdot 2^{-k}}\right)^{-\frac{1}{2^k-2}}+o(H)\\
&\ll H(y H^{1+\{\kappa\}})^{-\frac{1}{2^{k}-2}}+o(H)=o(H).
\end{align*}
Thus, we complete the proof of Lemma \ref{lem3.2}.
\end{proof}

We now verify that the sequence $PS(\kappa)$ with $\kappa\in\rb_+\setminus \nb$ satisfies the condition (\ref{2}).
For $(\log H)^{-2\kappa}\ll y\le 1/2$, using Lemma \ref{lem3.1} and Lemma \ref{lem3.2}, we have
\begin{align*}
\sum_{1\le \ell\le H}\|\lfloor\ell^{\kappa}\rfloor y\|^2&\ge \frac{1}{\pi^2 }\sum_{1\le \ell\le H}\sin^2\left(\pi^2 \lfloor\ell^{\kappa}\rfloor y\right)\\
&\ge \frac{1}{2\pi^2}\Re \sum_{H/2\le \ell\le H}\left(1-e^{2\pi\ri \lfloor\ell^{\kappa}\rfloor y}\right)\\
&\ge \frac{H}{4\pi^2}-\frac{1}{2\pi^2}\left|\sum_{H/2<\ell\le H}e^{2\pi\ri \lfloor\ell^{\kappa}\rfloor y}\right|+O(1)\\
&=\frac{H}{4\pi^2}-\frac{1}{2\pi^2}\cdot\frac{\pi H}{8}+o(H)=\frac{H}{4\pi^2}\left(1-\frac{\pi}{4}\right)+o(H),
\end{align*}
as $H\rrw \infty$. Note that $1-\pi/4>0$, which implies that for $(\log H)^{-2\kappa}\ll y\le 1/2$, one has
\begin{align}\label{eq1}
\frac{1}{\log H} \sum_{1\le \ell\le H}\|\lfloor\ell^{\kappa}\rfloor y\|^2&\gg \frac{H}{\log H}\rrw \infty, \quad H \rrw \infty.
\end{align}
On the other hand, for $(2\lfloor H^{\kappa}\rfloor)^{-1}\le y\le (\log H)^{-2\kappa}$, using the fact that $\| x\|=x$ for $x\in(0,1/2)$, we have the estimate that
\begin{align*}
\frac 1 {\log H} \sum_{1\le \ell\le H}\|\lfloor\ell^{\kappa}\rfloor y\|^2&\gg \frac 1 {\log H} \sum_{1\le \ell\le (2y)^{-1/\kappa}}\ell^{2\kappa} y^2\gg \frac {y^{-1/\kappa}} {\log H} \gg (\log H)\rrw\infty,
\end{align*}
as $H\rrw \infty$. Therefore, combining with \eqref{eq1}, we obtain that the sequence $PS(\kappa)$ with $\kappa\in\rb_+\setminus \nb$ satisfies the condition (\ref{2}).

\section{The proof of Proposition \ref{pro1}}\label{sec4}
In this section, we use a theorem of van der Corput \cite{MR1544956}, the classical Taylor's theorem to study the analytic properties of $\zeta_\kappa(s)$.
\subsection{On a theorem of van der Corput}

Define the Dirichlet series
\begin{align*}
\hat{\zeta}_{\alpha}(s)=\sum_{n\ge 1}\frac{\widetilde{B}_1(n^{\alpha})}{n^s},
\end{align*}
for all $\Re(s)>1$. Recall that $\widetilde{B}_1(x)=\{x\}-1/2$, we have $\hat{\zeta}_{\alpha}(s)=-\zeta(s)/2$ when  $\alpha$ is a positive integer.  For $\alpha\in\rb_+\setminus \nb$, we prove the following lemma.
\begin{lemma}\label{lemma41}Let $\alpha\in\rb_+\setminus \nb$. We have
$$\sum_{1\le n\le N}\widetilde{B}_1(n^{\alpha})\ll N^{1-\sigma_{\alpha}}\log N,\ \text{for}\ N\ge 2,$$
where
$$\sigma_{\alpha}=\max_{m\in\nb_{\ge 2}}\min\left(2^{1-m}, \frac{m-\alpha}{2^m-1}, 2^{1-m}\alpha\right)\in (0,1).$$
\end{lemma}
\begin{proof}
Fixed $\alpha\in\rb_+\setminus \nb$, let $f(x)=x^{\alpha}, x>0$. Then for any given integer $m\ge 2$ and positive number $X\ge 2$ we have $f^{( m )}(x)\asymp X^{\alpha-m}$, for all $x\in(X, 2X]$. Using a theorem of van der Corput, see \cite[Theorem (van der Corput)]{MR43119} or \cite[Satz 3]{MR1544956}, we have
\begin{align*}
\sum_{X<n\le 2X}\widetilde{B}_1(n^{\alpha})&\ll X\left(X^{-2^{1-m}}\log X+X^{\frac{\alpha-m}{2^m-1}}+(X^{m}X^{\alpha-m})^{-2^{1-m}}\right).
\end{align*}
Choosing integer $m\ge 2$ such that the exponent of $X$ above takes the smallest value, we obtain
\begin{align*}
\sum_{X<n\le 2X}\widetilde{B}_1(n^{\alpha})&\ll X^{1-\sigma_{\alpha}}\log X.
\end{align*}
Finally we replace the dyadic segment $X<n\le 2X$ by the whole interval $1\le n\le N$ by subdividing latter, we have
$$\sum_{1\le n\le N}\widetilde{B}_1(n^{\alpha})\ll 1+\sum_{0\le j\le \log_2 N-1}\left|\sum_{2^{-j-1}N< n\le 2^{-j}N} \widetilde{B}_1(n^{\alpha})\right|\ll N^{1-\sigma_{\alpha}}\log N.$$
This completes the proof.
\end{proof}
Using integration by parts for a Riemann-Stieltjes integral, we obtain the following proposition.
\begin{proposition}\label{pro41}
Let $\alpha\in\rb_+\setminus \nb$ and let $\sigma_\alpha(\in(0,1))$ be given by Lemma \ref{lemma41}. Then $\hat{\zeta}_{\alpha}(s)$ can be holomorphic continued analytically to $\Re(s)>1-\sigma_{\alpha}$. Moreover,
$$\hat{\zeta}_{\alpha}(s)\ll_{\varepsilon}1+|s|,$$
for all $\Re(s)\ge\varepsilon+1-\sigma_{\alpha}$ with any sufficiently small $\varepsilon>0$.
\end{proposition}

\subsection{Analytic continuation of $\zeta_{\kappa}(s)$}
\begin{proposition}\label{lemm2} For $\alpha\in \nb$ we have
\begin{equation}\label{eqzkai}
\zeta_{\kappa}(s)=\sum_{0\le h< \alpha}\binom{\alpha}{h}\zeta(s-h).
\end{equation}
\end{proposition}
\begin{proof}Since $\alpha\in\nb$, the series for $\zeta_{\kappa}(s)$ can be rewritten as
\begin{align*}
\zeta_{\kappa}(s)&=\sum_{n\ge 1}\frac{\#\{\ell\in \nb: \lfloor \ell^{\kappa}\rfloor=n\}}{n^s}\\
&=\sum_{n\ge 1}\frac{(n+1)^{\alpha}-n^{\alpha}}{n^s}\\
&=\sum_{n\ge 1}\frac{1}{n^s}\sum_{0\le h< \alpha}\binom{\alpha}{h}n^{h},
\end{align*}
which completes the proof by interchanging the order of the above summation.
\end{proof}

\begin{proposition}\label{lemm3}For $\alpha\not\in \nb$, we have
\begin{equation}\label{eqra1}
\zeta_{\kappa}(s)=\sum_{1\le \ell\le \lceil\alpha\rceil}\frac{\Gamma(s+\ell)}{\ell!\Gamma(s)}\zeta(s+\ell-\alpha)-\frac{s}{2}\zeta(s+1)
-s\hat{\zeta}_{\alpha}(s+1)+sR_{\alpha}(s),
\end{equation}
where $R_{\alpha}(s)$ is given as below \eqref{eqra} and holomorphic function for all $\Re(s)>-\min(1/2, 1-\{\alpha\})$. Moreover, for all $\Re(s)\ge \varepsilon-\min(1/2, 1-\{\alpha\})$ with any small $\varepsilon>0$,
$$R_{\alpha}(s)\ll_{\varepsilon} (|s|+1)^{\lceil\alpha\rceil}.$$
\end{proposition}
\begin{proof}The series for $\zeta_{\kappa}(s)$ can be rewritten as
\begin{align*}
\zeta_{\kappa}(s)&=\sum_{n\ge 1}\frac{\#\{\ell\in \nb: \lfloor \ell^{\kappa}\rfloor=n\}}{n^s}\\
&=\sum_{n\ge 1}\frac{\lceil (n+1)^{\alpha}\rceil-\lceil n^{\alpha}\rceil}{n^s}\\
&=\sum_{n\ge 2}\left(\frac{1}{(n-1)^s}-\frac{1}{n^s}\right)(\lceil n^{\alpha}\rceil-1).
\end{align*}
Let us define that
\begin{align}
d_{\alpha}=
\begin{cases} \qquad\qquad\qquad\qquad \infty, \quad &\alpha\not\in \qb,\\
\text{the denominator of reduced fraction of}~\alpha,\quad &\alpha\in \qb.
\end{cases}
\end{align}
Since $\alpha\not\in\nb$, it is clear that $d_\alpha\ge 2$.  Using the fact that for all $x\in\rb$,
\begin{align*}
\lceil x\rceil-1=x-{\bf 1}_{x\in \zb}-\{x\},
\end{align*}
we have
\begin{align*}
\zeta_{\kappa}(s)&=\sum_{n\ge 2}\left(\left(1-\frac 1n\right)^{-s}-1\right)\frac{ n^{\alpha}-{\bf 1}_{n^{\alpha}\in \zb}-\{n^{\alpha}\}}{n^s}\\
&=\sum_{n\ge 2}\frac{\left(1-\frac 1n\right)^{-s}-1}{n^{s-\alpha}}-\sum_{n\ge 2}\frac{\left(1-\frac {1}{n^{d_{\alpha}}}\right)^{-s}-1}{n^{d_{\alpha}s}}-\sum_{n\ge 2}\frac{\left(\left(1-\frac 1n\right)^{-s}-1\right)\{n^{\alpha}\}}{n^s}.
\end{align*}
Using the classical Taylor's theorem, we have
$$\left(1-t\right)^{-s}=\sum_{0\le j\le h}\frac{(s)_jt^j}{j!}+\frac{(s)_{h+1}t^{h+1}}{h!}\int_{0}^{1}\frac{(1-\theta)^{h}\,d\theta}{(1-t\theta)^{s+h}},\quad |t|<1,$$
for each integer $h\ge 0$, where $(s)_{\ell}=s(s+1)\cdots(s+\ell-1)=\Gamma(s+\ell)/\Gamma(s)$ is the Pochhammer symbol. Thus, we have
\begin{align*}
\zeta_{\kappa}(s)=\sum_{n\ge 1}\frac{1}{n^{s-\alpha}}\sum_{1\le \ell\le \lceil\alpha\rceil}\frac{(s)_{\ell}}{\ell !}\frac{1}{n^{\ell}}+s\sum_{n\ge 1}\frac{1/2-\{n^{\alpha}\}}{n^{s+1}}-\frac{s\zeta(s+1)}{2}+sR_\alpha(s),
\end{align*}
where
\begin{align}\label{eqra}
R_\alpha(s)=&-\sum_{n\ge 2}\frac{1}{n^{d_{\alpha}(s+1)}}\int_{0}^{1}\frac{\,d\theta}{(1-\theta/n^{d_\alpha})^{s+1}}-\sum_{n\ge 2}\frac{(s+1)\{n^{\alpha}\}}{n^{s+2}}\int_{0}^{1}\frac{(1-\theta)\,d\theta}{(1-\theta/n)^{s+2}}\nonumber\\
&+\sum_{n\ge 2}\frac{(s+1)_{\lceil\alpha\rceil}}{n^{s-\alpha+ \lceil\alpha\rceil+1}\lceil\alpha\rceil!}
\int_{0}^{1}\frac{(1-\theta)^{\lceil\alpha\rceil}\,d\theta}
{(1-\theta/n)^{s+\lceil\alpha\rceil+1}}-\sum_{1\le \ell\le \lceil\alpha\rceil}\frac{(s+1)_{\ell-1}}{\ell !}.
\end{align}
It is clear that $R_{\alpha}(s)$ holomorphic for all $\Re(s)>-\min(1-\{\alpha\}, 1/2)$.
Moreover, for $\sigma=\Re(s)\ge \varepsilon-\min(1-\{\alpha\}, 1/2)$ with any $\varepsilon>0$, we have
\begin{align*}
|R_\alpha(s)|&\ll \sum_{n\ge 2}\frac{1}{n^{d_{\alpha}(\sigma+1)}}+\sum_{n\ge 2}\frac{(|s|+1)}{n^{\sigma+2}}+\sum_{n\ge 2}\frac{(|s|+1)_{\lceil\alpha\rceil}}{n^{\sigma-\alpha+ \lceil\alpha\rceil+1}}+\sum_{0\le \ell< \lceil\alpha\rceil}\frac{(|s|+1)_{\ell}}{(\ell+1) !}\\
&\ll \zeta(2\sigma+2)+(|s|+1)\zeta(\sigma+2)+(|s|+1)_{\lceil\alpha\rceil}\zeta(\sigma+2-\{\alpha\})\\
&\ll_{\varepsilon} (|s|+1)_{\lceil\alpha\rceil}\ll (|s|+1)^{\lceil\alpha\rceil}.
\end{align*}
 This completes the proof of this proposition.
\end{proof}

It is well-known that $\zeta(s)$ can be monomorphic continued analytically to whole complex plane $\cb$ with only pole $s=1$. From Proposition \ref{pro41}, $\hat{\zeta}_{\alpha}(s)$ can be holomorphic continued analytically to $\Re(s)>1-\sigma_{\alpha}$.
Therefore, from Proposition \ref{lemm2} and Proposition \ref{lemm3}, we find that there exists
$$\sigma_{\kappa}=\min(\sigma_\alpha, 1/2, 1-\{\alpha\})\in(0,1]$$
such that $\zeta_{\kappa}(s)$ can be monomorphic continued analytically to $\Re(s)>-\sigma_{\kappa}$ and all poles lies on $\Re(s)>-\sigma_{\kappa}$ which are simple listed as follows
$$\alpha, \alpha-1, \ldots, \alpha+1-\lceil \alpha\rceil.$$
Combing with \eqref{eqzkai} and \eqref{eqra1}, it is clear that $$\res_{s=\alpha-h}\left(\zeta_\kappa(s)\right)=\frac{\Gamma(\alpha+1)}{(h+1)!\Gamma(\alpha-h)},$$
and for all $\Re(s)\ge\varepsilon-\sigma_{\kappa}$  with any $\varepsilon>0$ and $|\Im(s)|\ge 1$,
$$\zeta_{\kappa}(s)\ll_\varepsilon |s|(|s|+1)^{\lfloor \alpha\rfloor+1}.$$
Then, under the above discussion, the proof of of Proposition \ref{pro1} will follows from below subsection about the evaluating of $\zeta_{\kappa}(0)$ and $\zeta_\kappa'(0)$.

\subsection{The evaluating of $\zeta_{\kappa}(0)$ and $\zeta_\kappa'(0)$}
\begin{proposition}\label{lemm4} For $\alpha\in \nb$ we have
$$\zeta_{\kappa}(0)=-\frac{\alpha}{\alpha+1}\;\quad\mbox{and}\;\quad \zeta_{\kappa}'(0)=\sum_{0\le h< \alpha}\binom{\alpha}{h}\zeta'(-h).$$
\end{proposition}
\begin{proof}For $\alpha\in \nb$, from \eqref{eqzkai}
$$\zeta_\kappa(s)=\sum_{0\le h< \alpha}\binom{\alpha}{h}\zeta(s-h),$$
and the well-known facts that $\zeta(0)=-1/2$ and $\zeta(-h)=-B_{h+1}/(h+1), h\in\nb$, where $B_{\ell}$ is the usual $\ell$-th Bernoulli number, we have
\begin{align*}
\sum_{0\le h< \alpha}\binom{\alpha}{h}\zeta(-h)&=-\frac{1}{2}-\sum_{1\le h< \alpha}\binom{\alpha}{h}\frac{B_{h+1}}{h+1}\\
&=-\frac{1}{2}-\frac{1}{1+\alpha}\left(\sum_{0\le h\le \alpha}\binom{\alpha+1}{h}B_{h}-1-\binom{\alpha+1}{1}B_{1}\right)\\
&=-\frac{1}{2}-\frac{1}{1+\alpha}\cdot\frac{\alpha-1}{2}=-\frac{\alpha}{\alpha+1}.
\end{align*}
Here we use the fact that $B_1=-\frac 12$ and for $\alpha\in\nb$,
$$\sum_{0\le h\le \alpha}\binom{\alpha+1}{h}B_{h}=0.$$
This completes the proof.
\end{proof}
\begin{lemma}
For $\alpha\not\in \nb$ we have
 $$\zeta_{\kappa}(0)=-1/2$$
 and
\begin{align}
\zeta_{\kappa}'(0)=\sum_{1\le h\le \lceil\alpha\rceil}\frac{\zeta(h-\alpha)-1}{h}+\sum_{n\ge 2}\log\left(\frac{E_0(1/n)^{\widetilde{B}_1(n^{\alpha})+{\bf 1}_{n^{\alpha}\in\nb}}}{E_{\lceil\alpha\rceil}(1/n)^{n^{\alpha}}}\right).
\end{align}
\end{lemma}
\begin{proof}  It directly follows from \eqref{eqra1} and  the fact that $\zeta(s+1)=1/s+\gamma+O(s)$ that $\zeta_\kappa(0)=-1/2$ and
\begin{align}\label{eqrra}
\zeta_\kappa'(0)=\sum_{1\le \ell\le \lceil\alpha\rceil}\frac{\zeta(\ell-\alpha)}{\ell}-\frac{\gamma}{2}-\hat{\zeta}_{\alpha}(1)+R_\alpha(0).
\end{align}
Here $\gamma$ is the Euler-Mascheroni constant.
From \eqref{eqra}, we have
\begin{align*}
R_\alpha(0)=&\sum_{n\ge 2}\log\left(1-\frac{1}{n^{d_\alpha}}\right)-\sum_{n\ge 2}\frac{\{n^{\alpha}\}}{n^{2}}\int_{0}^{1}(1-\theta)(1-\theta/n)^{-2}\,d\theta\nonumber\\
&+\sum_{n\ge 2}\frac{1}{n^{2-\{\alpha\}}}
\int_{0}^{1}(1-\theta)^{\lceil\alpha\rceil}(1-\theta/n)^{-\lceil\alpha\rceil-1}\,d\theta-\sum_{1\le \ell\le \lceil\alpha\rceil}\frac{1}{\ell}.
\end{align*}
Note that for each integer $h\ge 0$,
\begin{align*}
\int_{0}^{1}\frac{(1-\theta)^{h}\,d\theta}{(1-\theta/n)^{h+1}}=-n^{h+1}\bigg(\sum_{1\le j\le h}\frac{1}{jn^j}+\log\left(1-\frac{1}{n}\right)\bigg),
\end{align*}
we further have
\begin{align*}
R_\alpha(0)=&\sum_{n\ge 2}\log\left(1-\frac{1}{n^{d_\alpha}}\right)+\sum_{n\ge 2}\{n^{\alpha}\}\left(\frac{1}{n}+\log\left(1-\frac{1}{n}\right)\right)\nonumber\\
&-\sum_{n\ge 2}n^{\lceil\alpha\rceil-1+\{\alpha\}}
\bigg(\sum_{1\le j\le \lceil\alpha\rceil}\frac{1}{jn^j}+\log\left(1-\frac{1}{n}\right)\bigg)-\sum_{1\le \ell\le \lceil\alpha\rceil}\frac{1}{\ell}.
\end{align*}
Using the fact that
$$\sum_{n\ge 2}\left(\frac{1}{n}+\log\left(1-\frac{1}{n}\right)\right)=\gamma-1,$$
it implies that
\begin{align*}
R_\alpha(0)=&\sum_{n\ge 2}\log\left(1-\frac{1}{n^{d_\alpha}}\right)+\sum_{n\ge 2}(\{n^{\alpha}\}-1/2)\log\left(1-\frac{1}{n}\right)+\frac{\gamma}{2}+\hat{\zeta}_{\alpha}(1)\nonumber\\
&-\sum_{n\ge 2}n^{\alpha}
\bigg(\sum_{1\le j\le \lceil\alpha\rceil}\frac{1}{jn^j}+\log\left(1-\frac{1}{n}\right)\bigg)-\sum_{1\le \ell\le \lceil\alpha\rceil}\frac{1}{\ell}.
\end{align*}
Recall that $E_h(z)=(1-z)e^{\sum_{1\le j\le h}z^j/j}$ and $\widetilde{B}_1(x)=\{x\}-1/2$ and we can complete the proof.
\end{proof}
\paragraph{Acknowledgements.} Nian Hong Zhou was partially supported by the promotion project of basic ability for young and middle-aged teachers in Universities Province of Guangxi (2021KY0064).  Ya-Li Li was supported by the National Natural Science Foundation of China (Grant Nos. 11901156 and 11771211).



\bigskip
\noindent
{\sc Nian Hong Zhou\\
School of Mathematics and Statistics, Guangxi Normal University\\
No.1 Yanzhong Road, Yanshan District, Guilin, 541006\\
Guangxi, PR China}\newline
Email:~\href{mailto:nianhongzhou@outlook.com; nianhongzhou@gxnu.edu.cn}{\small nianhongzhou@outlook.com; nianhongzhou@gxnu.edu.cn}

\bigskip
\noindent
{\sc Ya-li Li\\
School of Mathematics and Statistics, Henan University\\
Kaifeng 475001\\
Henan, PR China}\newline
Email:~\href{njliyali@sina.com}{\small njliyali@sina.com}

\end{document}